\documentclass[12pt,a4paper]{amsart}
\usepackage{amssymb}
\usepackage{amsmath}
\usepackage{amsthm}
\usepackage{mathrsfs}
\usepackage{epsfig}
\usepackage{color}

\newtheorem{proposition}{Proposition}
\newtheorem{remark}{Remark}
\numberwithin{equation}{section}

\def\Z{\mathbb Z}
\theoremstyle{plain}

\newcommand{\R}{\mathbb{R}}
\newcommand{\N}{\mathbb{N}}
\newcommand{\T}{\mathbb{T}}

\DeclareMathOperator\supp{supp}

\def\D{\,\mathrm d}
\def\e{\mathrm e}

\def\supp{\mathop{\rm supp}\nolimits}

\def\spn{\mathop{\rm span}\nolimits}
\def\dim{\mathop{\rm dim}\nolimits}

\def\D{\mathrm{d}}

\begin{document}

\title[Convolution and product in finitely generated shift-invariant spaces]{Convolution and product in finitely generated shift-invariant spaces}

\author[A. Aksentijevi\' c ]{Aleksandar Aksentijevi\' c }
\address{University of Kragujevac,
Faculty of Science, Department of Mathematics and Informatics, Radoja Domanovi\' ca 12,
34000 Kragujevac,
Serbia}
\email{aksentijevic@kg.ac.rs}

\author[S. Aleksi\' c]{Suzana Aleksi\' c}
\address{University of Kragujevac,
Faculty of Science, Department of Mathematics and Informatics, Radoja Domanovi\' ca 12,
34000 Kragujevac,
Serbia}
\email{suzana.aleksic@pmf.kg.ac.rs}

\author[S. Pilipovi\' c]{Stevan Pilipovi\'{c}}
\address{University of Novi Sad, Faculty of Sciences, Department of Mathematics and Informatics,  Trg Dositeja
Obradovica 4, 21000 Novi Sad, Serbia}
\email{pilipovic@dmi.uns.ac.rs}

\maketitle

\begin{abstract}
We discuss some structural properties of finitely generated shift-invariant (FGSI) spaces and subspaces of Sobolev spaces, particularly those related to convolution and the product within these spaces. We find shift-invariant solutions in FGSI spaces for a class of differential-difference equations with constant coefficients. Additionally, we analyze the Fourier multipliers in  FGSI spaces and the wave fronts for the convolution and product in FGSI spaces.
\end{abstract}

\maketitle

%%%%%% THIS PART MUST BE PLACED IMMEDIATELY AFTER THE \maketitle COMMAND
%%%%%% BACK TO ORIGINAL FOOTNOTES
\makeatletter
\renewcommand\@makefnmark%
{\mbox{\textsuperscript{\normalfont\@thefnmark)}}}
\makeatother
%%%%%%

\section{Introduction}

Following the approach of \cite{AST}--\cite{BVR2}, \cite{R}, to shift-invariant subspaces of the $L^2=L^2(\mathbb R^d)$ space,
later developed by many authors (see  \cite{A1}, \cite{A2}, \cite{MB}, \cite{BR},
and \cite{ss}), in this paper, we investigate mainly the convolution in
 shift-invariant  subspaces of Sobolev spaces $H^s=H^s(\mathbb R^d)$, $s\in\mathbb R,$ denoted as $V_s$. These subspaces are generated by a finite family of generators,
$\{\phi_1,\ldots,\phi_m\}\subset H^s$, such that
\[V_s(\phi_1,\ldots,\phi_m)=\Big\{f\in H^s: f=\sum_{i=1}^m\sum_{k\in\Z^d}c^i_k\phi_i(\cdot-k)\Big\}\] is the closure of the span of integer translations of the functions $\phi_1,\ldots,\phi_m$, where
$m\in\N$. 

In our investigations, we combine the approaches of \cite{BVR1} and \cite{AST} in case $s=0$ and deal with the shift-invariant subspaces $V_s$ of $H^s$ and $\mathcal V_s$ of $L^2_s$.  Actually, for $s\geqslant0$, the additional assumptions in the generators $\phi_1,\ldots,\phi_m$ (condition $(A)$ below)  imply that $V_s=\mathcal V_s$. In this way, we simplify our analysis of convolution in such spaces. Our results are of different nature in comparison to the ones of \cite{AST2}, where the convolution/deconvolution problem (a bounded or non-bounded inverse for the convolution) related to the sampling theory was studied mainly by the use of Gramian matrices. Also, our approach is different to that one of \cite{jia} where the solvability of linear operator equations is considered with applications to local shift-invariant spaces (see \cite{AS}, for example). We are interested in relations between various $V_{s}$ spaces related to the convolution inside such spaces, and mainly in the Fourier multipliers for such spaces, which are also shift-invariant. The structure of convolutors and the continuity properties are investigated. In addition, the product and the wave fronts for both are. With our approach, we also have results related to deconvolution, the solution of a convolution equation (differential-difference equation in Proposition \ref{eh}), or the deconvolution of an equation in which the generators are of the form $\mathcal F^{-1}(a)*\phi_i$, $i=1,\ldots,m$, where $a$ is a Michlin-type Fourier multiplier (Remark \ref{R11}).

The reader can see more about the spaces $V_s$ in \cite{aap1}--\cite{aap3}.
In the case $s=0$, one arrives at the well-known theory in $L^2$,  the space $V = V_0$.
The structure of shift-invariant spaces $V$ in Euclidean space $\R^d$ is well-established and has been widely applied in approximation theory, sampling theory, and wavelet theory.
The basic papers \cite{BVR1} and \cite{BVR2} focus on the
multivariate approximation properties of closed shift-invariant spaces, providing approximating families in $L^2$ that are
constructed by shifts of a well-understood set of functions.
Since the publication of these papers, interest in such spaces has continued to grow. The structural properties of shift-invariant spaces and
shift-preserving operators are further developed in
\cite{MB} and \cite{BR} through the analysis of the range function, which contributes to the relevant theory of subspaces of $\ell^2$, frames
and Riesz families through Gramians and dual Gramians. In \cite{AST}, these spaces are related to $L^p$ theory in such a way that the closedness of these spaces determines the frame properties of their generators. Aldroubi and his collaborators contributed much to the analysis and applications of FGSI spaces in various directions, investigating convolution \cite{AST2}  and locally finite-dimensional shift-invariant spaces \cite{AS}, always motivated by sampling theory. In several papers, we consider such spaces for $p=2$ in the Sobolev space framework and denote them as $\mathcal V_s$ (\cite{aap1}). From the perspective of sampling theory, an important paper is \cite{AG}, which relates nonuniform sampling to these spaces. Furthermore, deep extensions were made in \cite{A1} and \cite{A2}, where the authors further explore the diagonalization of finitely generated shift-invariant spaces (FGSI spaces) into a finite sum of closed, shift-invariant subspaces, on which the action of shift-invariant operators is well understood. Moreover, they characterized
shift-invariant operators whose finite iterations generate such spaces.
The H\"ormander result related to shift-invariant operators from $L^p$ into $L^q$ spaces, as well as the theory of multipliers (thus a class of pseudodifferential operators), is used in the analysis of FGSI spaces in connection with convolution and products connected by convolution via the Fourier transform.

After reviewing known results on periodic distributions and sequence spaces, we recall the definition and properties of wave front sets in Section 2. In Section 3, we provide a structural characterization of elements in the dual of an FGSI space, using the fact that $V_s(\phi_1,\ldots,\phi_m)=\mathcal V_s(\phi_1,\ldots,\phi_m)$. This fact is also used throughout the remaining sections. Section 4 is devoted to the relations of FGSI spaces. The main part of the paper is Section~5, where the convolution in FGSI spaces is investigated. In this section, applications of some classical results are presented. It is shown that a certain class of equations with constant coefficients and delays, with a right-hand side in an FGSI space, has a solution in a certain FGSI space.
We conclude Section 5 by characterizing the convolutors (Fourier multipliers) for FGSI spaces, studying the product, and determining the wave fronts for the convolution and the product of elements in FGSI spaces.

 \section{Notation}

 For distribution spaces and the properties of wave front sets, we refer to \cite{H} and \cite{LS}. By $\mathcal D'(\Omega)$, where $\Omega\subseteq\mathbb R^d$ is open, is denoted the Schwartz space of distributions, which is the strong dual of the space of compactly supported smooth functions $\mathcal D(\Omega)$. The space $\mathcal E'(\Omega)$ is a subspace of $\mathcal D'(\Omega)$, consisting of compactly supported distributions. We also recall that $\mathcal S'=\mathcal S'(\R^d)$ is the space of tempered distributions, which is the strong dual of the space of rapidly decreasing functions $\mathcal S=\mathcal S(\R^d)$.

 Let $s\in\R$, $\T^d=[0,1)^d$, and $e_y(x):=\e^{-2\pi \sqrt{-1} \langle x,y\rangle}$, where $\langle x,y\rangle=\sum_{i=1}^dx_iy_i$, for $x,y\in\mathbb{R}^d$.
 We use the notation $\mu_s(\cdot)=(1+|\cdot|^2)^{s/2}$ and consider the Sobolev spaces (see \cite{A}) \[H^s=H^s(\R^d)=\left\{f\in\mathcal{S}^\prime : \int_{\R^d}|\widehat{f}(t)|^2\mu_{2s}(t)\D t<+\infty\right\},\]
 where $\widehat{f}$ denotes the Fourier transform of the integrable function $f$ defined by $\mathcal{F}f(t)=\widehat{f}(t)=\int_{\mathbb{R}^d}f(x)e_{t}(x)\D x$, $t\in\mathbb{R}^d$.
 Note, $\mathcal F^{-1}f(t)=\widehat f(-t)$.
 Sobolev spaces $H^s$ are Hilbert spaces with inner product $\langle f,g\rangle_{H^s}=\int_{\R^d}\widehat{f}(t)\overline{\widehat{g}(t)}\mu_{2s}(t)\D t.$
These spaces satisfy $H^s\subseteq L^2\subseteq H^{-s}$ for $s\geqslant 0$ (with $H^0=L^2$).
The weighted Lebesgue space $L^2_s=L^2_s(\R^d)=\left\{f: \int_{\R^d}|{f}(x)|^2\mu_{2s}(x)\D x<+\infty\right\}$ is also a Hilbert space with the corresponding inner product. Subsequently, ${\ell_s^p=\ell_s^p(\mathbb{Z}^d)=\big\{(c_k)_{k\in\mathbb{Z}^d}\,:\,\sum_{k\in\mathbb{Z}^d}|c_k|^p\mu_{ps}(k)<+\infty\big\}}$, $p\geqslant1$,
and ${\ell_s^\infty=\ell_s^\infty(\mathbb{Z}^d)=\big\{(c_k)_{k\in\mathbb{Z}^d}\,:\,\sup_{k\in\mathbb{Z}^d}|c_k|\mu_{s}(k)<+\infty\big\}}$.
The spaces $\ell_s^2$ with the inner product $\langle (c_k)_{k\in\mathbb{Z}^d},(d_k)_{k\in\mathbb{Z}^d}\rangle_{\ell_s^2}=\sum_{k\in\mathbb{Z}^d}c_k\overline{d_k}\mu_{2s}(k)$ are Hilbert spaces.

The convolution of two functions is defined as $f*g(x)=\int_{\mathbb R^d}f(t)g(x-t)\D t$, $x\in\R^d$, and for distributions it is given by $f*g(x)=( f(x)g(t),\phi(x+t))$, $x,t\in\R^d$. Both definitions hold in their respective domains, provided that the convolution exists in those domains.
If $f\in\mathcal D'(\mathbb R^d)$ and $(x_0,\xi_0)\in\Omega\times(\mathbb R^d\setminus \{0\})$, then this point is said to be microlocally regular for $f$ if there exists $\theta\in\mathcal D(U_{x_0})$, which is non-zero at $x=x_0$, where $U_{x_0}$ is an open neighborhood of $x_0$, and a cone neighborhood $\Gamma_{\xi_0}$ of $\xi_0$ such that for every $N>0$ there exists $C_N>0$ satisfying $|\mathcal F(f\theta)(\xi)|\leqslant \frac{C_N}{{\mu_{N}(\xi)}}$, $\xi\in\Gamma_{\xi_0}$.
The set of all microlocally regular points of $f$ forms an open set and its complement in $\Omega\times(\mathbb R^d\setminus \{0\})$ is called the wave front set, denoted by $WF f.$
Recall, if $f\in\mathcal D'(\Omega)$ and $g\in \mathcal E'(\Omega),$ then
\begin{equation}\label{star}
WF (f*g)\subseteq\big\{(x+y,\xi)\,:\, (x,\xi)\in WF f, (y,\xi)\in WF g\big\}.
\end{equation}
The product of $f$ and $g$ in $\mathcal D'(\Omega)$ is a local property. It exists in a neighborhood $U_{x_0}$ of $x_0$ if there exists a compactly supported function $\theta\in\mathcal D(U_{x_0})$ (non-zero at $x_0$) such that $\xi_0\in\mathbb R^d$ does not satisfy both $(x_0,\xi_0)\in WF f$ and $(x_0,-\xi_0)\in WF g$. In this case,
\begin{equation}\label{2star}WF (fg)\subseteq\big\{(x,\xi+\eta):  (x,\xi)\in WF f \mbox{ or } \xi=0\mbox{ and }  (x,\eta)\in WFg
 \mbox{ or }\eta=0\big\}.
\end{equation}
   \subsection{Periodic test function spaces}

 The space of periodic test functions $\mathscr{P}=\mathscr{P}(\R^d)$ consists of smooth periodic functions of the form $f=\sum_{k\in\Z^d}f_ke_k$, where the coefficients $f_k$ satisfy $\sum_{k\in\Z^d}|f_k|^2 \mu_{2n}(k)<+\infty$ for every $n\in\N_0=\mathbb{N}\cup\{0\}$
 ($f_k=\int_{\T^d}f(x)e_{-k}(x)\D x$, $k\in\Z^d$); $\T^d=[0,1)^d$ is the torus. The topology on $\mathscr{P}$ is induced by the sequence of norms $\|f\|_n=\sup_{x\in\T^d, |\alpha|\leqslant n}|f^{(\alpha)}(x)|$, $n\in\mathbb{N}_0$.
 The dual space of $\mathscr P$, known as the space of periodic distributions, is denoted by $\mathscr{P}^\prime$. The following characterization holds:
 $f=\sum_{k\in\Z^d}{f_k e_k} \in \mathscr{P}^\prime$ if and only if $\sum_{k\in\Z^d}|f_k|^2\mu_{-2n_0}(k)<+\infty$ for some $n_0>0$. % We use notation $\mathscr P'^{n_0}$ when this holds.
  If $f=\sum_{k\in\Z^d}{f_k e_k}\in\mathscr{P}^\prime$ and $g=\sum_{k\in\Z^d}{g_k e_k}\in\mathscr{P}$,
 then their dual pairing is given by $\left(f,g\right)_{\mathscr{P}^\prime,\mathscr{P}}=\sum_{k\in\Z^d}f_k g_k$.

 The equivalent definitions of periodic test functions can be formulated using weighted $\ell^{p}_{r}$ norms, as follows. We denote by $\mathscr P^{p,r}$, where $p\in[1,+\infty]$ and $r\in\R$, the space of elements $h\in\mathscr{P}^\prime$ such that $h=\sum_{k\in\mathbb Z^d}a_ke_{k}$, where $(a_k)_{k\in\Z^d}\in\ell^{p}_r$. Note that $\bigcap_{r\geqslant0}\mathscr P^{p,r}=\mathscr P$ and $\bigcup_{r\leqslant0}\mathscr P^{p,r}=\mathscr P^\prime$, for every $p\in[1,+\infty]$. Recall that if $f_1=\sum_{k\in\Z^d}f_k^1e_k\in\mathscr P^{p_1,r}$ and $f_2=\sum_{k\in\Z^d}f_k^2e_k\in\mathscr P^{p_2,r}$, then their product is defined as $f=f_1f_2=\sum_{k\in\Z^d}f_ke_k\in\mathscr P^{p,r}$,
  where $\frac{1}{p_1}+\frac{1}{p_2}=\frac{1}{p}+1$ and $f_k=\sum_{j\in\Z^d}f_{k-j}^1f_j^2$, $k\in\Z^d$. Moreover, the mapping $f_1f_2:\mathscr P^{p_1,r}\times\mathscr P^{p_2,r}\to\mathscr P^{p,r}$ is continuous. Furthermore, if $f_1\in\mathscr P^{p_1,r_1}$ and $f_2\in\mathscr P^{p_2,r_2}$, where $r_1+r_2\geqslant0$, then the mapping $f_1f_2:\mathscr P^{p_1,r_1}\times\mathscr P^{p_2,r_2}\to\mathscr P^{p,r}$,
  where $\frac{1}{p_1}+\frac{1}{p_2}=\frac{1}{p}+1$ and $r\leqslant\min\{r_1,r_2\}$, is continuous (see \cite{mp}).

\section{FGSI spaces}
  Let $s\in\mathbb R$ and $\phi_i\in H^s$, $i=1,\ldots,m$. Following the approach of \cite{BVR1}-\cite{BR}, we  define the FGSI space $V_s(\phi_1,\ldots,\phi_m)$ in \cite{aap1} (see also \cite{aap2} and \cite{aap3}) as
   $$V_s(\phi_1,\ldots,\phi_m)=\overline{\spn}\{\phi_i(\cdot-k):k\in\mathbb{Z}^d, i=1,\ldots,m\}.$$ The spaces $V_s(\phi)$ are called principal shift-invariant (PSI) spaces.

Let $$\phi_i\in L^2_s\cap \mathcal L^\infty, i=1,\ldots,m, \mbox{ where } \mathcal L^\infty=\bigg\{f\in L^\infty: \sup_{x\in\T^d}\sum_{j\in\Z^d}|f(x+j)|<+\infty\bigg\}.$$  We denote by  $\mathcal V_s(\phi_1,\ldots,\phi_m)$, $s\geqslant 0,$ the space of functions of the form $$f=\sum_{i=1}^m\sum_{k\in\mathbb{Z}^d}b^i_k\phi_i(\cdot-k), \;\mbox{where}\; (b^i_k)_{k\in\mathbb{Z}^d}\in\ell^2_s,\, i=1,\ldots,m.$$ It holds that
  $\mathcal V_s(\phi_1,\ldots,\phi_m)\subseteq\spn\{\phi_i(\cdot-k): k\in\mathbb{Z}^d, i=1,\ldots,m\}\subseteq V_s(\phi_1,\ldots,\phi_m)$,
  where the middle space is densely embedded into the one on the right.
  By \cite{aap1}, for $\phi_i\in H^s\cap L^2_s\cap\mathcal L^\infty$, $i=1,\ldots,m$, we have $V_s(\phi_1,\ldots,\phi_m)=\mathcal V_s(\phi_1,\ldots,\phi_m)$
  if both $\mathcal V_s(\phi_1,\ldots,\phi_m)$ and $\mathcal{F}\big(\mathcal V_s(\phi_1,\ldots,\phi_m)\big)$ are closed in $L_s^2$. We refer to the following condition as (A):
  \begin{align*}(A)\quad & s\geqslant 0,\; \phi_i \in  H^s\cap L^2_s\cap \mathcal L^\infty,
  i=1\ldots,m, \,\mbox{ and }\;\\
  &\mbox{ both }\mathcal V_s(\phi_1,\ldots,\phi_m)\mbox{ and }\mathcal{F}\big(\mathcal V_s(\phi_1,\ldots,\phi_m)\big) \mbox{ are closed in }L_s^2.
  \end{align*}
  Recall \cite{aap1} (see also \cite{AST} for case $s=0$) that the space $\mathcal V_s(\phi_1,\ldots,\phi_m)$ is closed if and only if the set $\{\phi_i(\cdot-k) : k \in \mathbb{Z}^d,\ i=1,\ldots,m\}$ is a frame for $\mathcal V_s(\phi_1,\ldots,\phi_m)$. Moreover, this condition is equivalent to the fact that the same system is also a frame for $V_s(\phi_1,\ldots,\phi_m)$, provided that $\phi_i \in H^s \cap L_s^2 \cap \mathcal{L}^\infty$ for all $i = 1,\ldots,m$. In this case, if $f=\sum_{i=1}^m\sum_{k\in\mathbb{Z}^d}a^i_k\phi_i(\cdot-k)\in V_s(\phi_1,\ldots,\phi_m)$ and $\mathcal{F}(\mathcal V_s(\phi_1,\ldots,\phi_m))$ is closed, then the coefficient sequences $(a_k^i)_{k \in \mathbb{Z}^d}$ belong to $\ell_s^2$ for each $i = 1,\ldots,m$ (see \cite{aap1}). Therefore, under condition (A), we have the equality
\[V_s(\phi_1,\ldots,\phi_m)=\mathcal V_s(\phi_1,\ldots,\phi_m).\]
Additionally, if $(A)$ holds, then following \cite{BVR1}, $\{\phi_1,\ldots,\phi_m\}$ is called a basis for $V_s(\phi_1,\ldots,\phi_m)$ if, in the representation of $f$,
  \begin{equation}\label{re111}\widehat f=\sum_{i=1}^m\tau_i\widehat \phi_i,
  \end{equation}
  the periodic functions $\tau_i$ (with coefficients in $\ell^2_s$) are uniquely determined and defined on the torus $\mathbb{T}^d$.
Next, $\{\phi_1,\ldots,\phi_m\}$ is called a semi-basis for $V_s(\phi_1,\ldots,\phi_m)$ if in the expansion (\ref{re111})
 the periodic functions
$\tau_i$, $i=1,\ldots,m,$ are unique in the set $B$ of positive measure, which is a subset of the torus $\mathbb T^d$. According to \cite{BVR1}, $B$ equals the spectrum
$\sigma=\{x\in \mathbb T^d:\dim J_{V_s(\phi_1,\ldots,\phi_m)}(x)>0\}$ of $V_s(\phi_1,\ldots,\phi_m)$, where $J_{V_s(\phi_1,\ldots,\phi_m)}$ is the
range function (see \cite{MB} for $s=0$ and \cite{aap1}). Recall that $\sigma=\mathbb T^d$
for the basis. The analysis of such spaces for $s=0$ is provided in \cite{BVR1} and \cite{MB}.

\subsection{Duals for FGSI spaces}
Let $s\geqslant 0.$ We introduce
\begin{equation}\label{starr}\mathcal V_{-s}(f_1,\ldots,f_m)=
\bigg\{F=\sum_{j=1}^m\sum_{k\in\Z^d}
 c^j_kf_j(\cdot-k): (c^j_k)_{k\in\Z^d}\in \ell^2_{-s},
 f_j\in H^{-s}, j=1,\ldots,m\bigg\}.\end{equation}
 This is an ad hoc definition. It is not related to the definition of spaces $V_s$ for $s<0$. This space is used only for the analysis of the duality in this section.
 We say that $\{f_1,\ldots,f_m\}$ is a basis for $\mathcal V_{-s}(f_1,\ldots,f_m)$ if in the representation
 \begin{equation} \label{bsb}\widehat F=\sum_{j=1}^m\tau^j\widehat f_j,
\end{equation}
 periodic distributions $\tau^j(t)=\sum_{k\in\Z^d}c_k^je_{k}(t)$, $j=1,\ldots,m,$
 are uniquely determined for all $t\in\mathbb T^d$.
\begin{proposition}  $(a)$ Let $\{\phi_1,\ldots,\phi_m\}$ satisfy condition $(A)$ and let them be a basis for $V_s(\phi_1,\ldots,\phi_m)$. Assume that $\mathcal B$ is the set of all
elements  $(f_1,\ldots,f_m)\in(H^{-s})^m$ such that
\begin{equation}\label{usl11}\{f_1,\ldots,f_m\} \; \mbox{ is a basis for }\; \mathcal V_{-s}(f_1,\ldots,f_m).
\end{equation}
Then,
\begin{equation} \label{main}
(V_s(\phi_1,\ldots,\phi_m))'=
 \bigoplus_{(f_1,\ldots,f_m)\in \mathcal B}\mathcal V_{-s}(f_1,\ldots,f_m)
 \end{equation}
$(\bigoplus$ denotes the direct sum$)$.

$(b)$ Let $\{\phi_1,\ldots,\phi_m\}$ satisfy condition $(A)$ and let them be a semi-basis for $V_s(\phi_1,\ldots,\phi_m)$. Assume that $\mathcal B$ is the set of all
elements $(f_1,\ldots,f_m)\in(H^{-s})^m$ such that
for $F\in\mathcal V_{-s}(f_1,\ldots,f_m)$, in the representation \eqref{bsb}, the periodic distributions $\tau^j$, $j=1,\ldots,m,$ are uniquely determined on a set $B\subset \mathbb T^d$ of positive measure, the same as on which the periodic functions $\tau_i$, $i=1,\ldots,m,$ in \eqref{re111}, are uniquely determined. 
Then,
\begin{equation} \label{main22}
(V_s(\phi_1,\ldots,\phi_m))'=
 \bigoplus_{(f_1,\ldots,f_m)\in \mathcal B}\mathcal V_{-s}(f_1,\ldots,f_m).
\end{equation}
\end{proposition}
\begin{proof}
$(a)$ Let
 $ \mathcal V_{-s}(f_1,\ldots,f_m)$ be of the form (\ref{starr}),
where $(f_1,\ldots,f_m)\in \mathcal B.$
 The dual pairing  of $F\in H^{-s}$ and $\theta\in H^s$ is given by $(F,\theta)=F(\overline\theta)=\langle F,\overline \theta\rangle.$ Thus, if $F$ is as above
and $\theta=\sum_{i=1}^m\sum_{k\in \Z^d}a^i_k\phi_i(\cdot-k)$, then the following holds:
\begin{equation}\label{for1}
\big(F(x),\theta(x)\big)=\sum_{i=1}^m\sum_{j=1}^m\sum_{n\in\Z^d}
\sum_{p+k=n\atop p,k\in\Z^d}\overline a^i_{-k}c^j_p
\int_{\R^d}e_{n}(t)
\frac{\widehat{f_j}(t)}{\mu_s(t)}\overline{\widehat \phi_i}(-t)\mu_s(t)\D t.
\end{equation}
Recalling that $\big(f(x),\phi(x)\big)=\big(\widehat f(t),\widehat\phi(-t)\big)$, we express the dual pairing as
$$\big(f(x-p),\phi(x-k)\big)=\langle e_{p+k}(t)\widehat f(t),\overline{\widehat \phi}(-t)\rangle,\quad p,k\in\Z^d.
$$
The left-hand side of (\ref{for1}) is finite since it is bounded above by
$$ C\sum_{i=1}^m\sum_{j=1}^m\sum_{n\in\Z^d}\sum_{p+k=n\atop p,k\in\Z^d}a_{-k}^ic^j_p<C_1.
$$
 Using (\ref{usl11}), we identify
 $\mathcal V_{-s}(f_1,\ldots,f_m)$ with $(\ell^2_{-s})^m$ through
 $$V_{-s}(f_1,\ldots,f_m)\ni F=\sum_{j=1}^m\tau^j\widehat f_j\leftrightarrow (\tau^1,\ldots,\tau^m)\leftrightarrow ((c_p^1)_{p\in\Z^d},\ldots,(c_p^m)_{p\in\Z^d})\in(\ell^2_{-s})^m.
 $$
Similarly, we identify
 $V_{s}(\phi_1,\ldots,\phi_m)$ with $(\ell^2_{s})^m$ via
 $$V_s(\phi_1,\ldots,\phi_m)\ni \theta=\sum_{i=1}^m\tau_{i}\widehat \phi_i \leftrightarrow (\tau_1,\ldots,\tau_m)\leftrightarrow ((a_k^1)_{k\in\Z^d},\ldots,(a_k^m)_{k\in\Z^d})\in(\ell^2_{s})^m.
 $$
 Since $(\ell^2_{-s})^m$ is the dual space of $(\ell^2_{s})^m,$ we obtain (\ref{main}).

  $(b)$ The proof is the same as for $(a)$.
  % while for c) we use the fact that the dual of the
\end{proof}
\begin{remark}
Since every FGSI space is a finite orthogonal sum of FGSI spaces with semi-basis $($see \cite{BVR1}$)$, and the dual of such a space is the direct product of the corresponding duals for each space in the orthogonal sum,  it follows that the dual of any FGSI space can be determined as the product of corresponding orthogonal sums.
\end{remark}

   \section{Relations between FGSI spaces}

By \cite{BVR1}, we know that if $\phi,\psi\in L^2$ is such that $\psi\in V(\phi)$ and $\supp\widehat{\phi}\subseteq\supp \widehat{\psi},$ then $V(\psi)=V(\phi)$ and $\supp \widehat \phi=\supp\widehat \psi$.
We now give a more general proposition.

For this proposition, we need some preparation. Let $A$ be an
$m\times m$ matrix with elements $\tau_i^j$, $i,j\in\{1,\ldots,m\},$ and suppose that these elements are in $\mathscr P^{2,s}$. The matrix $A$ is said to be regular if there exists a matrix $B$ with elements $\eta_i^j$ such that $AB=BA=I$,  where $I$ is a diagonal matrix with elements $\theta _i^i\in \mathscr P^{\infty,s}$, $i=1,\ldots,m$.

  We denote the transposition by $\cdot^T$; for example, if $(\widehat \psi_1,\ldots,\widehat \psi_m)$ is a row vector, then $(\widehat \psi_1,\ldots,\widehat \psi_m)^T$ denotes the corresponding column vector.

\begin{proposition}  Let $\{\phi_1,\ldots,\phi_m\}$ satisfy condition $(A)$ and let them be a basis of $V_s(\phi_1,\ldots,\phi_m)$. Let $\psi_j\in V_s(\phi_1,\ldots,\phi_m)$, $j=1,\ldots,m$,
such that the matrix A satisfies
$$(\widehat \phi_1,\ldots,\widehat \phi_m)^T=A(\widehat \psi_1,\ldots,\widehat \psi_m)^T,$$
where the elements of the matrix $A$ are periodic functions
$$\tau_i^j=\sum_{k\in\Z^d}a^j_{i,k}e_{k}, \; i, j \in\{1,\ldots,m\},\mbox{ with }\; (a_{i,k}^j)_{k\in\Z^d}\in\ell^2_s,\;  i,j\in\{1,\ldots,m\}.$$
 Then, the following conditions $(a)$ and $(b)$ are equivalent:
 \begin{itemize}\item[$(a)$] $V_s(\phi_1,\ldots,\phi_m)=V_s(\psi_1,\ldots,\psi_m);$
 \item[$(b)$] $A$ is a regular matrix.
 \end{itemize}
 In both cases it follows that $\bigcup_{i=1}^m\supp \widehat \phi_i=\bigcup_{j=1}^m\supp \widehat \psi_j.$
Moreover, if the following conditions hold:
  \begin{itemize}
 \item[$(i)$] $\bigcup_{i=1}^m\supp \widehat \phi_i=\bigcup_{j=1}^m\supp \widehat \psi_j$, and
\item[$(ii)$]  $\supp \widehat \phi_{j}\setminus \bigcup_{i=1, i\neq j}^m
 \supp \widehat \phi_i \neq \emptyset,\;  \;j=1,\ldots,m,
$
\end{itemize}
 then the combination of $(i)$ and $(ii)$ is equivalent to $(a)$ and $(b)$.
 \end{proposition}

\begin{proof} Since $\psi_j\in V_s(\phi_1,\ldots,\phi_m)$, $j=1,\ldots,m$, there exists an $m\times m$ matrix $B$ with elements $\nu_i^j\in \ell^2_s$, $i,j\in\{1,\ldots,m\},$ such that
$$
(\widehat \psi_1,\ldots,\widehat \psi_m)^T=B(\widehat \phi_1,\ldots,\widehat \phi_m)^T.
$$
Now, the equivalence of $(a)$ and $(b)$ follows from the  fact that
$$(\widehat \phi_1,\ldots,\widehat \phi_m)^T=AB(\widehat \phi_1,\ldots,\widehat \phi_m)^T,$$
and the fact that $AB=I$ in case $(a)$ equivalently, in case $(b)$.
In both cases, for every
$i=1,\ldots,m$, it holds that
$$\supp \widehat \phi_i\subseteq \bigcup_{j=1}^m\supp \widehat \psi_j,$$
because if, for some $i_0,$
$$\supp \widehat \phi_{i_0}\setminus \bigcup_{j=1}^m\supp \widehat \psi_j\neq \emptyset,$$
 there exists an open set where
 $\widehat \phi_{i_0}\neq 0$, but $\widehat \psi_j=0$ for all $j=1,\ldots,m$, which would imply $\phi_{i_0}\notin V_s(\psi_1,\ldots,\psi_m)$, a contradiction.
Similarly, we have
$$\supp \widehat \psi_j\subseteq \bigcup_{i=1}^m\supp \widehat \phi_i,$$
completing the first part of the statement.

Now, assume that $(i)$ and $(ii)$ hold, but $(b)$ does not hold.  From the first part, we know that
$$ (\widehat \phi_1,\ldots,\widehat \phi_m)^T=AB(\widehat \phi_1,\ldots,\widehat \phi_m)^T,
$$
 but since $A$ is not regular,  we have $AB\neq I.$ Denote by $\theta_i^j$, $i,j \in\{1,\ldots,m\}$, the elements of the matrix $AB.$
 Since $AB\neq I$, there must exist some $\theta_i^j\in\ell^{\infty}_s$ that are non-zero for $i\neq j$ or, if such elements do not exist, there must exist some $\theta_i^i$ equal zero.
In either case,  there is some $j_0\in\{1,\ldots,m\}$ such that
\begin{equation}\label{c11}
\widehat \phi_{j_0}=\sum_{i=1\atop i\neq j_0}^m\theta_i^{j_0}\widehat \phi_i\quad \mbox{ or } \quad \widehat \phi_{j_0}=0.
\end{equation}
However, this leads to a contradiction with condition $(ii)$. %, because there would be an open set where $\phi_{j_0}\neq 0$, but the right-hand side of either case in (\ref{c11}) would be zero.
 Thus, the assumption that $A$ is not regular leads to a contradiction. Therefore, $A$ must be regular. This completes the proof of the statement.
\end{proof}

%Recall that if $f\in L^2_s$, then $f\in L^1$ for $s>\frac{d}2$, and also $\ell^1\subset\ell_s^2$. Moreover, if $s\leqslant r$, then $\ell_r^p\subseteq\ell_s^p$ for every $p\in[1,+\infty]$.

In the following, we denote the Dirac delta distribution by $\delta(\cdot)$. Recall that:\begin{itemize}\item$\delta\in H^s$ only if $s<-\frac{d}2$;
\item if $s>\frac{d}2$, then $L^2_s\subset L^1$ and $\ell_s^2\subset\ell^1$;
\item if $s\leqslant r$, then $\ell_r^p\subseteq\ell_s^p$ for every $p\in[1,+\infty]$.
 \end{itemize}
 \begin{proposition}\label{t1} Let $s\geqslant0$, $\phi_i\in H^s $, $i=1,\ldots,m$, and let
 $$f=\sum_{k\in\mathbb Z^d}a_k\delta(\cdot-k),\; (a_{k})_{k\in\Z^d}\in\ell_{r}^{1}\;\mbox{ and }\;g=\sum_{i=1}^m\sum_{{k\in\Z^d}}b^i_k\phi_i(\cdot-k),\;(b_k^i)_{k\in\Z^d}\in\ell_{s}^{2}.$$
 If $s\leqslant r$, then $f*g\in V_s(\phi_1,\ldots,\phi_m)$.
 \end{proposition}
 \begin{proof} It suffices to consider the case $m=1$.
 Let $f=\sum_{k\in\mathbb Z^d}a_k\delta(\cdot-k)$, $(a_{k})_{k\in\Z^d}\in\ell_{r}^{1}$, and $g=\sum_{k\in\mathbb Z^d}b_k\phi(\cdot -k)$, $(b_k)_{k\in\Z^d}\in\ell_s^2$. Since $s\leqslant r$, we get
 \begin{align*} \widehat{f*g}&=\widehat{f}\cdot\widehat{g}=\bigg(\sum_{k\in\Z^d}a_ke_{k}\bigg)\cdot\bigg(\sum_{k\in\Z^d}b_ke_{k}\widehat{\phi}\bigg)
 =\widehat{\phi}\bigg(\sum_{k\in\Z^d}a_ke_{k}\bigg)\cdot\bigg(\sum_{k\in\Z^d}b_ke_{k}\bigg)\\
 &=\widehat{\phi}\bigg(\sum_{k\in\Z^d}c_ke_{k}\bigg)=\sum_{k\in\Z^d}c_ke_{k}\widehat{\phi}=\Big({\sum_{k\in\Z^d}c_k\phi(\cdot-k)}\Big)^{\widehat{}},
  \end{align*}
  where $(c_k)_{k\in\Z^d}\in\ell_s^2$. Applying the inverse Fourier transform, we obtain $f*g(\cdot)=\sum_{k\in\Z^d}c_k\phi(\cdot-k)$. Thus, we conclude that $f*g\in V_s(\phi)$.
 \end{proof}
Denote by $H^s_c$ the space of compactly supported elements of $H^s$. 
 \begin{proposition} \label{11} Let $s>-\frac{d}{2}$. Assume that $\phi\in H_c^s$ and $\phi*\phi(\cdot)=\sum_{k\in\Z^d}a_k\phi(\cdot-k)$. Then, $\phi=0$.
 \end{proposition}
 \begin{proof}  Since $\phi$ is compactly supported, its Fourier transform $\widehat \phi$ is an analytic function. Thus, the set of its zeros, denoted by $\Lambda$, is discrete.\\
 Applying the Fourier transform, we obtain $\widehat{\phi}(t)\widehat{\phi}(t)=\widehat{\phi}(t)\sum_{k\in\Z^d}a_ke_{k}(t)$, $t\in\R^d$.
This implies that
  $\widehat{\phi}(t)=\sum_{k\in\Z^d}a_ke_{k}(t)$, for $t\in\R^d\setminus \Lambda$. Let $H(t)$, $t\in\R^d$, be a distribution
$$H(t)=\widehat \phi(t)-\sum_{k\in\Z^d}a_ke_k(t), \quad t\in\R^d.
$$
By the assumptions and the fact that $\phi*\phi\in H^s,$ we conclude $H\in L^2_s$. Moreover, since $H(t)=0$ for $t\in \R^d\setminus \Lambda$, it follows that $\supp H\subseteq \Lambda$.
   This implies that $$H(t)=\sum_{t_i\in \Lambda}\sum_{j=0}^{k(t_i)-1}c_i^j\delta^{(j)}(t-t_i),\quad t\in\mathbb R^d,
$$ where $k(t_i)$ denotes the multiplicity of zero $t_i\in\Lambda$. However, for such a distribution to belong to $L^2_s$, all of its coefficients must vanish. Therefore, $H = 0$, and we conclude that $\widehat{\phi}(t)=\sum_{k\in\Z^d}a_ke_{k}(t)$, $t\in\mathbb R^d$. This implies that ${\phi}(x)=\sum_{k\in\Z^d}a_k\delta(x-k)$, $x\in\mathbb R^d$. Since $s>-\frac{d}2$ and $\phi$ is assumed to belong to $H^s$, such a distribution cannot belong to $H^s$. Hence, all coefficients $a_k$ must be zero, and therefore $\phi=0$.
% and $V_s$ can not contain shifts of delta. 
%Now, let $(a,b)\subset \mathbb R^d\setminus \mathbb{Z}^d$ be a $d$-dimensional interval such that $\widehat \phi(t)\neq 0$ for $t\in(a,b)$, and let
%   $\theta\in\mathcal{S}$ be such that $\widehat{\theta}\in\mathcal{D}\big((a,b)\big)$ and $\big(\phi(x),\theta(-x)\big)\neq 0$. Then, we have
% $$\big(\widehat{\phi}(t),\widehat{\theta}(t)\big)=\Big(\sum_{k\in\Z^d}a_ke_{k}(t),\widehat{\theta}(t)\Big),$$
%which implies
%$$ \big(\phi(x),\theta(-x)\big)=\Big(\sum_{k\in\Z^d}a_k\delta(x-k),\theta(-x)\Big).$$
%  It follows that $\big(\sum_{k\in\Z^d}a_k\delta(x-k),\theta(-x)\big)=0$, which contradicts the assumption that $\big(\phi(x),\theta(-x)\big)\neq0$. This completes the proof.
\end{proof}

\begin{proposition}\label{22} Let $s>-\frac{d}{2}$ and $\phi,\psi\in H^s$ such that $\phi*\psi\in H^s$ and both $\widehat{\phi}$ and $\widehat \psi$ are real analytic. Then,
 \begin{itemize}\item[$(a)$] $\phi*\psi\notin V_s(\phi);$
 \item[$(b)$] $V_s(\phi*\psi) \cap V_s(\phi)=\{0\}$.
 \end{itemize}
\end{proposition}
 \begin{proof} $(a)$ Assume that $\phi*\psi\in V_s(\phi)$. Then, we have $\widehat{\phi}\cdot\widehat{\psi}=\widehat{\phi}\tau$,
 where $\tau$ is periodic. Let $A=\{t:\widehat{\phi}(t)\neq0\}$.
 Since $\widehat{\phi}$ is analytic, the set $\Lambda=\{t:\widehat{\phi}(t)=0\}$ is discrete and thus has measure zero.
Therefore, $\widehat{\psi}(t)=\tau(t)$ for $t\in A$.
Consequently, $\widehat{\psi}(t)=\sum_{k\in\Z^d}a_ke_{k}(t)$, $t\in\R^d\setminus \Lambda$.  Using reasoning similar to that in Proposition \ref{11}, we conclude
  that $\psi(x)=\sum_{k\in\Z^d}a_k\delta(x-k)$, for $x\in\mathbb{R}^d$. This leads to a contradiction with the assumption that $\psi\in H^s, s>-\frac{d}2$.

 $(b)$ Let $f\in V_s(\phi*\psi)\cap V_s(\phi)$ and $f\neq 0$. Then, we have the following expansions:
 $$f(x)=\sum_{k\in\Z^d}a_k\phi*\psi(x-k)=\sum_{j\in\Z^d}b_j\phi(x-j),\quad x\in\mathbb{R}^d.$$  Taking the Fourier transform of both sides, we obtain
 $$\widehat{f}(t)=\widehat{\phi}(t)\widehat{\psi}(t)\sum_{k\in\Z^d}a_ke_{k}(t)=\widehat{\phi}(t)\sum_{j\in\Z^d}b_j e_{j}(t),\quad t\in\mathbb{R}^d.$$
 By the uniqueness of Fourier expansions, for each $p\in\Z^d$, we must have
 $$\widehat{\phi}(t)\widehat{\psi}(t)a_p=\widehat{\phi}(t)b_p,\quad t\in\mathbb{R}^d.$$
 If for some $p\in\Z^d$, $a_p\neq b_p$, and if $\widehat{\phi}(t)\neq0$ for some $t\in\mathbb{R}^d$, then it must follow that $\widehat{\psi}(t)\neq0$ for such $t\in\mathbb{R}^d$, because if $\widehat{\psi}(t)=0$, we would have $b_p=0$ for all $p\in\Z^d$.
Furthermore, for such $t$, we must have $$\widehat{\psi}(t)a_p=b_p,\quad p\in\Z^d.$$
This relationship holds for every  $t\in\mathbb{R}^d$ where $\widehat{\psi}(t)\neq0$.
However, for different points $t_1,t_2$ where $\widehat{\psi}(t_1)\neq0$ and
$\widehat \psi(t_1)\neq \widehat{\psi}(t_2)$, we would require
$$\widehat{\psi}(t_1)a_p=b_p\quad\text{and}\quad\widehat{\psi}(t_2)a_p=b_p, \quad p\in\Z^d.$$
\noindent This is clearly impossible. Therefore, we conclude that
 $V_s(\phi*\psi)\cap V_s(\phi)=\{0\}.$
 \end{proof}

 \section{Convolution equations}
There are many results in the literature concerning the convolution
that can be applied to FGSI spaces. We mention a result from \cite{E4}: for a given real analytic function $\psi$ and $\phi\in\mathcal E'(\R^d)$, there exists $u\in\mathcal D'(\R^d)$ such that $u*\phi=\psi$. Note that, if $f\in\mathcal E^\prime(\R^d)$, then there exists $s\in\R$ such that $f\in H^s$.

 We also mention another important result. First, we observe that, in general, for given $\phi, \psi \in \mathcal{S}$, there does not exist $\theta \in \mathcal{S}$ such that $\phi=\psi*\theta$ (however, in certain simple cases, such as when $\phi(x)=\e^{-a(1+ |x|^2)}$ and $\psi(x)=\e^{-b(1+|x|^2)^{r/2}}$, $x\in\R^d,$ $a, b>0$, $0<r<2$, this is trivially true). The next assertion is a consequence of the result of \cite{V}: $\mathcal S=\mathcal S*\mathcal S.$
 \begin{proposition}
 $(a)$  Assume that $\phi\in \mathcal E^\prime(\R^d)$ and  $\psi\in H^s$ is real analytic.  Then, there exists $u\in\mathcal D'(\R^d)$ such that % for every $s\geqslant 0,$
 $$V_s(\psi)%=\mathcal V_s(\psi) = u*\mathcal V_s(\phi)
 =u*V_s(\phi)\left(=\{u*f: f\in V_s(\phi)\}\right).
 $$
 
 $(b)$ For every $\phi_i\in\mathcal{S}$ there exist $\phi_1^i, \phi_2^i\in\mathcal{S}$, $i=1,\ldots,m$, such that
 $$V_s(\phi_1,\ldots,\phi_m) =V_s(\phi_1^1*\phi_2^1,\ldots,\phi_1^m*\phi_2^m).$$
 \end{proposition}
 \begin{proof} % By assumption (A), $V_s(\phi)=\mathcal V_s(\phi)$ and $V_s(\psi)=\mathcal V_s(\psi)$. Additionally,
 According to \cite{E4}, there exists $u\in\mathcal D'(\R^d)$ such that $\psi=u*\phi$. Thus, for given $f=\sum_{k\in \Z^d}a_k\phi(\cdot-k)\in V_s(\phi)$,
 $$u*f=\sum_{k\in \Z^d}a_ku*\phi(\cdot-k)=\sum_{k\in \Z^d}a_k\psi(\cdot-k).$$Therefore,
 part $(a)$ follows.

 The part $(b)$ follows from the fact that for every $\phi_i\in\mathcal{S}$ there exist $\phi_1^i,\phi_2^i\in\mathcal{S}$ such that $\phi_i=\phi_1^i*\phi_2^i$, $i=1,\ldots,m$.

 \end{proof}

The next result from \cite{E2} (see also \cite{E4})  is useful: if $\psi$ is a real analytic distribution and
\begin{equation}\label{equat}
T(\cdot)=\sum_{i=1}^m \sum_{j=0}^{q_i}a_{j}^{i}\delta^{(j)}(\cdot-b_{j}^{i}),\quad b_{j}^{i}\in \mathbb {R}^d,
\end{equation}
then there exists  $u\in\mathcal D'(\R^d)$  such that
$T*u=\psi$. Assume that $T$ satisfies an additional assumption:
\begin{itemize}
\item[$(E)$]
There exist $c>0$ and $n\in\R$ such that
$$c\mu_n(t)\leqslant |\widehat T(t)|=\left|\sum_{i=1}^m\sum_{j=0}^{q_i}a_{j}^{i}\left(2\pi \sqrt{-1}\right)^jt^je_{b_j^i}(t)\right|,\quad t\in\R^d.
$$
\end{itemize}
%\begin{example} Let $T(x)=\sum_{i=1}^m\sum_{j=0}^{q_i}a_i\delta^{(j)}(x-b)$, $b\in\R^d$, where
%P(\partial_x)=\sum_{i=1}^m\sum_{j=0}^{q_i}a_i\partial_x^{j}$ is a strongly elliptic differential operator with constant coefficients so that $|P(\xi)|\geq C, \xi\in\R^d.$.
%Then condition $(E)$ holds.
%\end{example}
With this, we have the next proposition.
\begin{proposition}\label{eh} Let $s\geqslant0$ and $h=\sum_{k\in\Z^d}c_k\psi(\cdot-k)$, where $\psi\in H^s$ is real analytic and $(c_k)_{k\in\Z^d}\in\ell_s^2$. If $T$ satisfies condition $(E)$, %define $h=\sum_{k\in\Z^d}c_k\psi(\cdot-k)$, where $(c_k)_{k\in\Z^d}\in \ell^2_{s}.$ then there exists  $s_1\in\R$ such that
then the equation
\begin{equation}\label{eq}
\sum_{i=1}^m \sum_{j=0}^{q_i}a_{j}^{i} u^{(j)}(\cdot -b_{j}^{i})=h
\end{equation}
has a solution $u\in H^{s_1}$, where $s_1=s+n$ $($with $n$ as in condition $(E))$. Moreover, all solutions of the equation belong to the space
$V_{s_1}(u)$.%, where $s_0=\min\{s,s_1\}$.
\end{proposition}
\begin{proof}
By the assumption on $h$, we have $\widehat{h}\in L^2_s$, and the solution $u$ is given by $u=\mathcal F^{-1}(\widehat h/\widehat T)$. By condition $(E)$, it follows that $u\in H^{s_1}$ for $s_1=s+n$. This solution is unique.
%%%%%%%%%%%%%%%%%%
%%%%%%%%%%%%%%%%%%%%%%%%%
Next, consider the system of equations
\begin{equation}\label{sist}
T*u(\cdot-k)=\psi(\cdot-k),\quad k\in\Z^d.
\end{equation}
If $u=u_0$ is a solution of $T*u_0(x)=\psi(x)$, then the system \eqref{sist} has the solution $u_k=u(\cdot-k)$ for all $k\in\Z^d$. %, because $T*u(x-k)=\psi(x-k),\; k\in\Z^d,\; x\in\R^d.$
Thus, we have
$$T*\sum_{k\in\Z^d}c_k u(x-k)=\sum_{k\in\Z^d}c_k T*u(x-k)=\sum_{k\in\Z^d}c_k\psi(x-k)=h(x),\quad x\in\R^d.$$
This completes the proof.
\end{proof}
\begin{remark}
A much more difficult question arises if we do not assume the condition $(E)$, or more generally, if we consider division problems in shift-invariant spaces. Such questions are related to the work of Ehrenpreis $($see \cite{E2}, \cite{E4} for example$)$ and the solvability of equations with constant coefficients in the framework of tempered distributions \cite{Hd}. These problems are not considered here.
\end{remark}

\subsection{Convolutors}
We recall that for any operator with bounded shift-invariance $A:L^p\to L^q$, with $p\leqslant q$,  there exists a tempered distribution $a$ such that $Af=a*f$ (see \cite{Hc}, Theorem~1.2). If $a\in L^\infty$ satisfies
\begin{equation}\label{multi}|\partial^\alpha a(t)|\leqslant C_\alpha |t|^{-\alpha}, \quad 0\neq t\in\R^d,\, \alpha\leqslant{[d/2]+1},\end{equation}
 where $[d/2]$ denotes the greatest integer less than or equal to $d/2$,
then the shift-invariant operator $T_a:L^p\to L^p$, defined via its action in the Fourier domain by $\mathcal{F}(T_af)=a\widehat{f}$, is bounded for all $p\in(1,+\infty)$.
This is Mikhlin's multiplier theorem (see \cite{M}). A significant extension of this result is provided in \cite{Hc} and has since been further developed by many authors (see \cite{G} and references therein). In the context of Sobolev spaces, the operator $T_a:H^s\to H^s$ is continuous for all $s\in\R$ whenever $a\in L^\infty$ satisfies (\ref{multi}).
Operators of the form $T_a$ are commonly referred to as Fourier multipliers or convolutors.

\begin{remark} \label{R11} A direct consequence is the following assertion. Let $(A)$ hold for $\phi_1,\ldots,\phi_m$.
 If $a\in L^\infty$ satisfies \eqref{multi}, then for every shift-invariant space $V_s(\phi_1,\ldots,\phi_m)$, the mapping
$$T_a: V_s(\phi_1,\ldots,\phi_m)\rightarrow
V_s\big(\mathcal{F}^{-1}(a)*\phi_1,\ldots,\mathcal{F}^{-1}(a)*\phi_m\big),\;
$$ defined by
$$f=\sum_{i=1}^m\sum_{k\in \Z^d}f^i_k\phi_i(\cdot-k)\mapsto T_af=\sum_{i=1}^m\sum_{k\in \Z^d}f^i_k(\mathcal{F}^{-1}(a)*\phi_i)(\cdot-k),$$
is continuous.
 \end{remark}
 In general, the situation with FGSI spaces is different. In order to simplify the question of the continuities of convolutors, in this section, we assume that condition $(A)$ for $\phi$ holds so that $V_s(\phi)=\mathcal V_s(\phi)$ and the question of continuity is transferred to the corresponding sequence spaces.
 Our goal in the following assertion is to characterize the distributions $f\in\mathcal S'$ such that, for every
$g\in V_s(\phi)$, we have
$f*g\in V_s(\phi)$ and the mapping $g\mapsto f*g$ is continuous from $V_s(\phi)$ to $V_s(\phi)$. A partial answer is given in the next proposition.

  \begin{proposition}\label{pr8}  Let $(A)$ hold for $\phi$ and the set of zeros of $\widehat{\phi}$ be  a discrete set:
\[B=\{t_p : \widehat\phi(t_p)=0,\, p\in\mathbb I\subseteq \N\}.\] Let $\mathcal A$ be the set of  distributions of the form
\begin{equation}\label{disth8}\mathcal A=\bigg\{P(t)=\sum_{p\in\mathbb{I}}\sum_{j=0}^{j_p-1}c_j^p\delta^{(j)}(t-t_p),\; t\in\mathbb{R}^d\bigg\},
\end{equation} where $j_p$ denotes the multiplicity of the zero $t_p\in B$.
Define a mapping $f:V_s(\phi)\rightarrow V_s(\phi)$ such that $g\mapsto f*g$, $g\in V_s(\phi)$.
Then,
$f$ is of the form $$f=\sum_{k\in \Z^d}a_k\delta(\cdot-k)+\mathcal{F}^{-1}(P),$$
 where $(a_k)_{k\in\Z^d}$ is determined by expansion $f*\phi=\sum_{k\in\Z^d}a_k\phi(\cdot-k)$, and $P\in\mathcal{A}$.
 Moreover,  for every $g\in V_s(\phi)$
  the mapping $g\mapsto f*g$ is continuous from $V_s(\phi)$ to $V_s(\phi)$.
 \end{proposition}
 \begin{proof} Assume that $f*\phi\in V_s(\phi)$. Then, there exists a sequence $(a_k)_{k\in\Z^d}\in\ell_s^2$ such that
 $$\widehat{f}(t)\widehat{\phi}(t)=\widehat{\phi}(t)\sum_{k\in\Z^d}a_ke_{k}(t),\quad t\in\mathbb{R}^d.$$
Define the distribution $ h(x)=f(x)-\sum_{k\in\Z^d}a_k\delta(x-k)$, 
$x\in\R^d.$ Then, $h\in\mathcal{S}^\prime$ and its Fourier transform satisfies \[\widehat h(t)=\widehat{f}(t)-\sum_{k\in\Z^d}a_ke_{k}(t), 
\quad t\in \R^d.\] Using the previous identity, we see that $\widehat h(t)=0$ for all $t\in \R^d\setminus B$. Since $B$ is a discrete
 set of zeros of $\widehat{\phi}$, it follows that 
 $\widehat h=P\in\mathcal{A}$, and so 
 $$f=\sum_{k\in \Z^d}a_k\delta(\cdot-k)+\mathcal{F}^{-1}(P).$$
For every $g=\sum_{k\in\Z^d}b_k\phi(\cdot-k)\in V_s(\phi)$ there holds $g*\mathcal{F}^{-1}(P)=0$, and therefore
$$f*g=\sum_{n\in\Z^d}\sum_{k+j=n\atop k,j\in\Z^d}a_kb_j\phi(\cdot-n).$$
This shows that $f*g\in V_s(\phi)$, and standard arguments imply that the mapping $g\mapsto f*g$ is continuous on $V_s(\phi)$.\end{proof}

The Proposition \ref{pr8} shows that the convolutors in $V_s(\phi)$ are essentially of the form $\sum_{k\in\Z^d}a_k\delta(\cdot-k)$, since the additional term (an element of $\mathcal{F}^{-1}(\mathcal A)$) annihilates over $V_s(\phi).$

\begin{remark} Let $\phi_1,\ldots,\phi_m$ satisfy $(A)$.
The same conclusion holds when the zero sets of $\widehat \phi_1,\ldots,\widehat \phi_m$, denoted by $B_i$ are discrete and the sets $\mathcal A_i$ are defined as in $(\ref{disth8})$, $i=1,\ldots,m$. Convolutors on the space $V_s(\phi_1,\ldots,\phi_m)$
 are of the form
 $$f=\sum_{i=1}^m\sum_{k^i\in\Z^d}a_{k^i}\delta(\cdot-k^i)+\sum_{i=1}^m\mathcal{F}^{-1}(P^i),\quad P^i\in\mathcal A_i,\, i=1,\ldots,m,
 $$
where the sequences $(a_{k^i})_{k^i\in\Z^d}$, $i=1,\ldots,m$, are determined by
 $f*\phi_i=\sum_{k^i\in\Z^d}a_{k^i}\phi_i(\cdot-k^i)$, $i=1,\ldots,m$.\\
 The validity of this statement follows as a direct consequence of the proof of Proposition~$\ref{pr8}$. That is, assuming that $f*g\in V_{s}(\phi_1,\ldots,\phi_m),$ for all $g\in  V_{s}(\phi_1,\ldots,\phi_m)$, it is sufficient to restrict the analysis to the subspace $V_{s}(\phi_1)$, where the conclusion of Proposition~$\ref{pr8}$ can be invoked.
 \end{remark}
\begin{remark}
A typical example of a generator $\phi$ is a function $\phi\in\mathcal D$. More generally, one may consider $\phi$ so that $(A)$ holds and $\widehat \phi$ is analytic.
\end{remark}

We now characterize the wave front in the setting of FGSI spaces. In general, let $\phi\in\mathcal E'(\mathbb R^d)$.
Then, for $f=\sum_{k\in\mathbb Z^d}a_k\phi(\cdot-k)$ (which belongs to $\mathscr D'$ for any sequence $(a_k)_{k\in\mathbb Z^d}$),
 \begin{equation}\label{wfcon}
WF f\subseteq\big\{(x+k,\xi) : (x,\xi)\in WF \phi, k\in\mathbb Z^d\big\}.
\end{equation}
This follows from the fact that $f=\phi*h$, where $h=\sum_{k\in\mathbb Z^d}a_k\delta(\cdot-k)$ with $WF h\subseteq\mathbb Z^d\times(\mathbb R^d\setminus\{0\})$.
\begin{proposition}\label{WFcon} Let $f\in V_s(\phi_1,\ldots,\phi_m)$ and $g\in V_r(\psi_1,\ldots,\psi_l)$, where $\phi_1,\ldots,\phi_m$ and $\psi_1,\ldots,\psi_m$ satisfy condition $(A)$ with $\frac{d}2< s\leqslant r$. Then, %{\color{red}If $\phi_1,\ldots,\phi_m$ or $\psi_1,\ldots,\psi_l$ have compact supports,} then
\begin{itemize}
\item[(1)] $f*g\in\mathcal V_{s-d/2-\varepsilon}(\phi_i*\psi_j,\,i=1,\ldots,m, j=1,\ldots,l)$ for $\varepsilon>0$ such that $s\geqslant\frac{d}2+\varepsilon$,
\item[(2)] the mapping $(f,g)\mapsto f*g$,
$$V_s(\phi_1,\ldots,\phi_m)\times V_r(\psi_1,\ldots,\psi_l)\rightarrow\mathcal V_{s-d/2-\varepsilon}(\phi_i*\psi_j,\,i=1,\ldots,m, j=1,\ldots,l),$$ is separately continuous,
\item[(3)]
$WF (f*g)\subseteq\bigcup\limits_{1\leqslant i\leqslant m\atop 1\leqslant j\leqslant l}\big\{(x+y+k,\xi) :\, (x,\xi)\in WF \phi_i,\, (y,\xi)\in WF \psi_j,\, k\in\mathbb Z^d\big\}.$
\end{itemize}
\end{proposition}
\begin{proof} (1) There holds
$$f*g=\sum_{i=1}^m\sum_{j=1}^l\sum_{p\in\Z^d}\sum_{k\in\Z^d}
a^i_kb^j_p\phi_i*\psi_j(\cdot-p-k)=
\sum_{n\in\Z^d}\sum_{k+p=n\atop k,p\in\Z^d}a_kb_p\phi(\cdot-n),
 $$
  where $(a^i_k)_{k\in\Z^d}*(b^j_p)_{p\in\Z^d}\in\ell^\infty_s$ and $\phi_i*\psi_j\in L_s^\infty$, $i=1,\ldots,m$, $j=1,\ldots,l$. Therefore, we conclude that $(a^i_k)_{k\in\Z^d}*(b^j_p)_{p\in\Z^d}\in\ell_{s-d/2-\varepsilon}^2$ and $\phi_i*\psi_j\in L^2_{s-d/2-\varepsilon}$, for $\varepsilon>0$ such that $s\geqslant\frac{d}2+\varepsilon$ and all $i=1,\ldots,m$, $j=1,\ldots,l$.
 This implies that
\begin{equation}\label{sepnep}f*g\in\mathcal V_{s-d/2-\varepsilon}(\phi_i*\psi_j, i=1,\ldots,m, j=1,\ldots,l),
\end{equation} for $\varepsilon>0$ such that $s\geqslant\frac{d}2+\varepsilon$. 

(2) If $((a_{k,n}^i)_{k\in\Z^d})_{n\in\Z^d}$ converges to $(a_k^i)_{k\in\Z^d}$ in $\ell^2_s$, $i=1\ldots,m$, then by (\ref{sepnep})
$$(a^i_{k,n})_{k\in\Z^d}*(b^j_p)_{p\in\Z^d}\rightarrow  (a^i_k)_{k\in\Z^d}*(b^j_p)_{p\in\Z^d}\quad\mbox{as }n\rightarrow\infty\mbox{ in }\ell^2_{s-d/2-\varepsilon},$$ for $\varepsilon>0$ such that $s\geqslant\frac{d}2+\varepsilon$. 
This proves the continuity. 
The same holds for the continuity with fixed $(a^i_k)_{k\in\Z^d}$, $i=1,\ldots,m$.

(3) The result follows from (\ref{star}) and (\ref{wfcon}).
\end{proof}
\begin{remark} $(i)$ The statement also holds for $s\geqslant0$, where 
\begin{align*}\mathcal V_{s-d/2-\varepsilon}(f_1,\ldots,f_m)=
\bigg\{F=\sum_{j=1}^m\sum_{k\in\Z^d}
 c^j_kf_j(\cdot-k):& (c^j_k)_{k\in\Z^d}\in \ell^2_{s-d/2-\varepsilon},\\
 & f_j\in L_{s-d/2-\varepsilon}^2\cap\mathcal L^\infty, j=1,\ldots,m\bigg\}.
 \end{align*}
 $(ii)$ If we omit the condition $(A)$, the statement will also hold for spaces $\mathcal V_s(\phi_1,\ldots,\phi_m)$ and $\mathcal V_r(\psi_1,\ldots,\psi_l)$, where $0\leqslant s\leqslant r$.
\end{remark}

\subsection{On the product in FGSI spaces}
If $\phi\in H^{s_1}$, $\psi\in H^{s_2}$, then it is known that the product
$\phi\psi$ belongs to $H^{s}$ under the conditions $s\leqslant s_1$, $s\leqslant s_2$, $s_1+ s_2\geqslant 0$, and $s_1+s_2-s\geqslant \frac{d}2$ (if $s_1=\frac{d}2$ or $s_2=\frac{d}2$ or $s=-\frac{d}2$, then the inequality is strict; see Theorem 8.3.1 in \cite{Hor2}).

The product $gf$ of elements belonging to two finitely generated shift-invariant (FGSI) spaces does not, in general, belong to any FGSI space. We aim to identify sufficient conditions under which the product belongs to some FGSI space.
If $f=\sum_{i=1}^m\sum_{k\in\mathbb Z^d}$ $a_k^i\phi_i(\cdot-k)\in V_{s}(\phi_1,\ldots,\phi_m)$, then the product $gf=\sum_{i=1}^m\sum_{k\in\mathbb Z^d}a_k^i(g\phi_i)(\cdot-k)$
should belong to a certain space $V_{s_0}(g\phi_1, \ldots, g\phi_m)$, provided that the relations $g(\cdot)\phi_i(\cdot-k)=(g\phi_i)(\cdot-k)$, $k\in \Z^d$, $i=1,\ldots,m$, hold.
These identities are satisfied when $g$ is a periodic function. However, periodicity alone is not sufficient. To ensure the required regularity and structural properties of the product, we impose the following assumption:
 \begin{equation}\label{zapr} \mbox{There exists }p\in \mathbb{N}_0 \text{ such }g^{(\alpha)}\in L^\infty \text{ are periodic functions for all } \alpha \in \mathbb{N}_0^d, \; |\alpha| \leqslant p. \end{equation}
Since $g$ and $\phi$ are functions whose product exists in $L^2$, we will treat their product as a product of functions and then analyze the wave front set of $gf$. A more complex situation arises when the product is considered in the sense of distributions and analyzed through \eqref{2star}, but we will not consider that case here.
%  Under this assumption, we have the following proposition.

\begin{proposition}\label{p10}
Let $(A)$ hold for $\phi_1,\ldots,\phi_m$ with $s\in\N_0$, and let assumption \eqref{zapr} hold for some $p\in\N_0$ and a function $g$. Let $f=\sum_{i=1}^m\sum_{k\in\mathbb Z^d}a_k^i\phi_i(\cdot-k)\in V_s(\phi_1,\ldots,\phi_m)$. Then the following statements are valid.
\begin{itemize}\item[(1)] The product $gf$ belongs to the space
$ V_{s_0}(g\phi_1,\ldots,g\phi_m), $ where $s_0=\min\{p,s\}$.
\item[(2)] The mapping
\begin{equation}\label{cop} V_s(\phi_1,\ldots,\phi_m)\ni f\mapsto gf=\sum_{i=1}^m\sum_{k\in\mathbb Z^d}a_k^i(g\phi_i)(\cdot-k)\in V_{s_0}(g\phi_1,\ldots,g\phi_m)
\end{equation}
is continuous.
\item[(3)] $WF(gf)\subseteq\bigcup\limits_{1\leqslant i\leqslant m}\big\{(x+k,\xi) : (x,\xi)\in WF (g\phi_i), k\in\mathbb Z^d\big\}$.
\end{itemize}
\end{proposition}
\begin{proof} (1) Since each $\phi_i\in H^s$, it follows that $(g\phi_i)^{(\beta)}\in L^2$, for every $\beta\in\mathbb N_0^d$, $|\beta|\leqslant s_0$, $i=1,\ldots,m$, by condition $\eqref{zapr}$. Thus, $g\phi_i\in H^{s_0}$, $i=1,\ldots,m$. %Also, $g\phi_i\in L_{s_0}^2\cap\mathcal L^\infty$, $i=1,\ldots,m$. 
Furthermore, $(a_k^i)_{k\in\Z^d}\in\ell^2_s\subseteq\ell^2_{s_0}$, $i=1,\ldots,m$, which implies that
$$gf=\sum_{i=1}^m\sum_{k\in\mathbb Z^d}a_k^i(g\phi_i)(\cdot-k)
\in V_{s_0}(g\phi_1,\ldots,g\phi_m).$$

(2) Let $((a_{k,n}^i)_{k\in\Z^d})_{n\in\Z^d}$ be a sequence converging to $(a_k^i)_{k\in\Z^d}$ in $\ell^2_s$ as $n\rightarrow \infty$, $i=1,\ldots,m$. Since the embedding $\ell^2_s\hookrightarrow\ell^2_{s_0}$ is continuous, the convergence also holds in $\ell^2_{s_0}$, which implies the continuity of the mapping in \eqref{cop}.

(3) The assertion follows directly from \eqref{wfcon}.
\end{proof}
%\begin{remark} If we omit both conditions \eqref{zapr} and $(A)$, the statements in Proposition $\ref{p10}$ will hold only for the spaces $V_s$, where $s_0=s$.
%\end{remark}

\vskip1cm

{\bf Acknowledgment.}
S. Pilipovi\' c is supported by the Serbian Academy of Sciences and
Arts, project F10. A. Aksentijevi\' c and S. Aleksi\' c were supported by the
Science Fund of the Republic of Serbia, $\#$GRANT No $2727$, {\it Global
and local analysis of operators and distributions}  (GOALS), and they appreciate the financial support of the Ministry of
Science, Technological Development and Innovation of the Republic
of Serbia (Grants No. 451-03-136/2025-03/200122 and 451-03-137/2025-03/200122).

\end{document}